\documentclass[12pt,a4paper]{amsart}
\usepackage{amsmath,enumerate,etoolbox,calc,booktabs,graphicx}
\usepackage[a4paper,centering]{geometry}
\usepackage[mathcal]{eucal}
\patchcmd{\section}{\scshape}{\bfseries}{}{}
\patchcmd{\subsubsection}{\itshape}{\bfseries}{}{}
\makeatletter
\renewcommand{\@secnumfont}{\bfseries}
\makeatother
\numberwithin{equation}{section}
\theoremstyle{plain}
\newtheorem{theorem}{Theorem}
\newtheorem{proposition}{Proposition}
\newtheorem*{proposition*}{Proposition}
\newtheorem{lemma}{Lemma}
\newtheorem{lemmma*}{Lemma}
\newtheorem{corollary}{Corollary}
\theoremstyle{definition}
\newtheorem{definition}{Definition}
\newtheorem*{definition*}{Definition}
\newcommand{\sgn}{{\rm sgn}}
\newcommand{\C}{\mathcal C}
\newcommand{\T}{\mathcal T}
\newcommand{\cross}{{\rm cr}}
\newcommand{\Irr}{{\rm Irr}}
\newcommand{\wt}{{\rm wt}}
\newcommand{\blambda}{\boldsymbol\lambda}
\newcommand{\bmu}{\boldsymbol\mu}
\newcommand{\Tab}{{\rm Tab}}
\newcommand{\sh}{{\rm sh}}
\newcommand\ink[1]{\begin{tabular}{@{}c@{}}#1\end{tabular}}
\begin{document}
% \ \vspace{-.55in}
\ \vspace{-.9in}
\title[Congruences in character tables]{\makebox[0mm]{Congruences in character tables of symmetric groups}}
\author[A.~R.~Miller]{Alexander~R.~Miller}
% \address{Fakult\"at f\"ur Mathematik, Universit\"at Wien, Vienna, Austria}
% \email{alexander.r.miller@univie.ac.at}
\begin{abstract}
  If $\lambda$ and $\mu$ are two non-empty Young diagrams
  with the same number of squares,
  and $\blambda$ and $\bmu$ are obtained by dividing each square  
  into $d^2$ congruent squares, then the corresponding character
  value $\chi_{\blambda}(\bmu)$ is divisible by~$d!$.
\end{abstract}
\maketitle
\thispagestyle{empty}
%%%%%%%%%%%%%%%%%%%%%%%%%%%%%%%%%%%%%%%%%%%%%%%%%%%%%%%%%%%%%%%%%%
%%%%%%%%%%%%%%%%%%%%%%%%%%%%%%%%%%%%%%%%%%%%%%%%%%%%%%%%%%%%%%%%%%
%%%%%%%%%%%%%%%%%%%%%%%%%%%%%%%%%%%%%%%%%%%%%%%%%%%%%%%%%%%%%%%%%%
\ \vspace{-.5in}

\section{Introduction}
For any partition $\lambda=1^{m_1}2^{m_2}\ldots n^{m_n}$ of an integer $n$,
let $\chi_\lambda$ be the corresponding irreducible character of the symmetric group $S_n$,
let $\chi_\lambda(\mu)$ be the value at any $\sigma\in S_n$ of cycle type $\mu$,
and, fixing once and for all a positive integer $d$, define partitions 
\[
  d.\lambda = d^{m_1}(2d)^{m_2}\ldots (nd)^{m_n},\qquad
  \blambda=d^{dm_1}(2d)^{dm_2}\ldots (nd)^{dm_n},
\]
so $d.\lambda$ is obtained by scaling the parts of $\lambda$, and 
$\blambda$ is obtained by subdividing the squares of the Young diagram of $\lambda$.
The purpose of this paper is to prove:

\begin{theorem}\label{Thm:main}
  For any two partitions $\lambda$ and $\mu$ of a positive integer, 
  \begin{equation}\label{Mod:Vanish}
    \chi_{\blambda}(\bmu)\equiv 0\pmod{d!}.
  \end{equation}
  More generally, 
  for any partition $\lambda$ of a positive integer $n$, and any 
  partition $\mu$ of $dn$,
    \begin{equation}\label{General:Mod:Vanish}
    \chi_{\blambda}(d.\mu)\equiv 0\pmod{d!}.
  \end{equation}
  For any two partitions $\lambda$ and $\mu$ of a positive integer not divisible by $d$, 
  \begin{equation}\label{Ord:Vanish}
    \chi_{\blambda}(d^2.\mu)=0.
  \end{equation}
\end{theorem}
Explicit results like these are rare. Previous results 
include J.~McKay's characterization of 
partitions $\lambda$ of $n$ satisfying 
${\chi_\lambda(1^n)\equiv 0\pmod 2}$ \cite{McKay}, I.~G.~Macdonald's
generalization for $\chi_\lambda(1^n)\equiv 0\pmod p$~\cite{Macdonald},
the corollary of Murnaghan--Nakayama that $\chi_\lambda(\mu)=0$
under certain conditions involving hook lengths \cite{MacdonaldBook}, 
and the relation between ordinary and modular vanishing
given by the fact that Frobenius' formula for
$\chi_\lambda(\mu)$ \cite{Frobenius} implies, for any prime~$p$,
that $\chi_\lambda(\mu)\equiv\chi_\lambda(\nu)\pmod p$
whenever $\nu$ can be obtained from $\mu$ by breaking 
some part into $p$ equal~parts.

There are also general results of Burnside, J.~G.~Thompson, and P.~X.~Gallagher, with
Burnside proving that zeros exist for nonlinear irreducible characters
of a finite group \cite{Burnside}, Thompson modifying Burnside's argument 
with a result of C.~L.~Siegel \cite{Siegel} to show that
each irreducible character
evaluates to zero or a root of unity on more than a third of the group elements \cite{Isaacs},
and Gallagher proving similarly
that more than a third of the irreducible characters 
are zero or a root of unity on any larger than average class \cite{Gallagher}.

A few years ago, 
it was shown that if $\chi\in\Irr(S_n)$ and $\sigma\in S_n$ are chosen at random, 
then $\chi(\sigma)=0$ with probability $\to 1$ as $n\to \infty$ \cite{Miller1}. 
The analogous result for ${\rm GL}(n,q)$ was established
in joint work of the author with Gallagher and Larsen \cite{GLM}.
Larsen and the author subsequently showed that the proportion of
zeros in the character table of a finite simple group of Lie type goes to
$1$ as the rank goes to infinity \cite{LM}. 
The limiting behavior for the proportion of zeros in the character table of $S_n$ is not yet
known, but it was conjectured in \cite{Miller2} that, for any prime~$p$,
the proportion of $p$-divisible entries in the character table of $S_n$ goes to $1$ as $n\to\infty$.\footnote{
  As suggested by the computations in \cite{Miller2},
  the author suspects that the same is true for arbitrary prime powers: if
  $\lambda$ and $\mu$ are chosen uniformly at random from the partitions of $n$, then for any prime power $q$,   
  $\chi_\lambda(\mu)\equiv 0\pmod q$ with probability $\to 1$ as $n\to\infty$.}
The classical results of McKay \cite{McKay} and Macdonald \cite{Macdonald} imply that
the proportion of $p$-divisible entries in the column of degrees $\chi_\lambda(1)$ goes to $1$ as $n\to\infty$,
and recent results of Gluck \cite{Gluck} and Morotti~\cite{Morotti} deal with certain other columns.
Very recently, Peluse \cite{Peluse} established the conjecture for
primes $\leq 13$, and then Peluse and Soundararajan \cite{PS} together
established the full conjecture for all primes. 
So for each prime $p$ we have $\chi_\lambda(\mu)\equiv 0\pmod p$ with probability $\to 1$ as $n\to\infty$.
Theorem~\ref{Thm:main} gives an unexpected stability result that answers the natural question of 
what happens if the shapes $\lambda$ and $\mu$ are naturally dilated with large scale factor: 
for any prime $p$, \eqref{Mod:Vanish} with $d\geq p$ 
implies that $\chi_{\blambda}({\boldsymbol\mu})\equiv 0\pmod p$ for {\it all} partitions $\lambda,\mu$ of a positive integer.

We prove \eqref{Mod:Vanish} and \eqref{General:Mod:Vanish}
by showing that
in the Murnaghan--Nakayama formula for
computing $\chi_{\blambda}(d.\mu)$
as a weighted sum over certain rim hook tableaux,
the relevant rim hook tableaux
admit an action of $S_d$ that is both free and weight-preserving.
This is done by first translating 
from rim hook tableaux to some new objects 
we call {\it cascades}, which are a matrix analogue of 
Com\'et's classical one-line binary notation for partitions, and which 
can be viewed as collections of lattice paths with weight defined in terms of crossings. 
As a benefit of independent interest,  we obtain 
a lattice-path version of Murnaghan--Nakayama in Proposition~\ref{MN:Lattice-Paths}.
Then in Theorem~\ref{Thm:Action} we 
establish an explicit weight-preserving free action of $S_d$
on cascades. As a corollary we obtain
\eqref{Mod:Vanish} and \eqref{General:Mod:Vanish}, 
while \eqref{Ord:Vanish} will come from Proposition~\ref{MN:Lattice-Paths}.

%%%%%%%%%%%%%%%%%%%%%%%%%%%%%%%%%%%%%%%%%%%%%%%%%%%%%%%%%%%%%%%%%%
%%%%%%%%%%%%%%%%%%%%%%%%%%%%%%%%%%%%%%%%%%%%%%%%%%%%%%%%%%%%%%%%%%
%%%%%%%%%%%%%%%%%%%%%%%%%%%%%%%%%%%%%%%%%%%%%%%%%%%%%%%%%%%%%%%%%%
\section{Preliminaries}\label{Sect:Preliminaries}
By {\it partition} of an integer $n\geq 0$ we mean an integer sequence 
${\lambda=(\lambda_1,\lambda_2,\ldots,\lambda_l)}$ satisfying  
${\lambda_1\geq \lambda_2\geq\ldots\geq \lambda_l\geq 1}$ and
$\lambda_1+\lambda_2+\ldots +\lambda_l=n$.
We say $\lambda$ has {\it size} $n$ with $l$ {\it parts}, writing $|\lambda|=n$ and $\ell(\lambda)=l$. 
The alternative shorthand $\lambda=1^{m_1}2^{m_2}\ldots n^{m_n}$ means 
$\lambda$ is the partition with $m_1$ $1$'s, $m_2$ $2$'s, and so on,
e.g.\ $(4,2,1,1)=1^22^14^1$. 

We identify $\lambda$ with its {\it shape} or {\it Young diagram}, i.e.\ 
the left-justified array with $\lambda_1$ squares in
the first row, $\lambda_2$ squares in the second row, and so on,
e.g.\ the partition $(8,6,4,3)$ is identified with the following shape:
\[\includegraphics{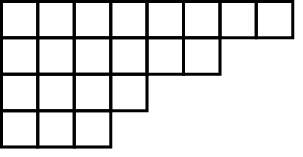}\]

By {\it rim hook} $\rho$ of $\lambda$ we mean the union of a non-empty
sequence of squares
in $\lambda$ such that each square is directly to the left or directly
below the previous square and $\lambda\setminus \rho$ is a Young diagram,
e.g.\ the following is a rim hook of size $7$ in $(8,6,4,3)$:
\[\includegraphics{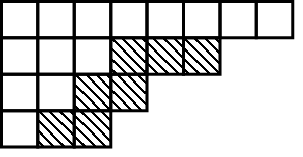}\]

By {\it rim hook tableau} $T$
we  mean a labeling of the squares of a
non-empty Young diagram $\lambda$ with 
integers $1,2,\ldots, m$ such that 
the squares with label $\geq i$ form a
Young diagram $T_i$ and, for $1\leq i\leq m$,
the squares labeled $i$ form a (non-empty) rim hook
of size $\alpha_i$ in~$T_i$.
We say $T$ has {\it shape} $\lambda$ and 
{\it content} $\alpha=(\alpha_1,\alpha_2,\ldots,\alpha_m)$,
we write \[T={\rm Tab}(T_1,T_2,\ldots,T_{m+1}),\]
so $T_1=\lambda$ and $T_{m+1}=\emptyset$, and we 
define the {\it weight} of $T$ by 
\begin{equation}
  \wt(T)=\prod_{i=1}^m (-1)^{\#\{\text{rows of $T$ occupied by }i\}-1}.
\end{equation}
An example rim hook tableau of shape $(8,6,4,3)$
and content $(4,4,6,3,2,2)$ is 
\[
\includegraphics{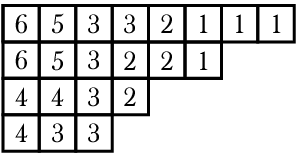}
\]
which has weight  $(-1)^{2-1+3-1+4-1+2-1+2-1+2-1}$. 

Denoting by $\T(\lambda,\alpha)$ the set of all rim hook tableaux
of shape $\lambda$ and content $\alpha=(\alpha_1,\alpha_2,\ldots,\alpha_m)$, 
the mapping
$T\mapsto (T_1,T_2,\ldots,T_{m+1})$
takes
$\T(\lambda,\alpha)$ bijectively onto the set of all partition sequences
$\lambda=\lambda^1,\lambda^2,\ldots ,\lambda^{m+1}=\emptyset$
in which each succeeding $\lambda^{i}$ is obtained 
from the previous partition $\lambda^{i-1}$ 
by removing a rim hook of size $\alpha_{i-1}$, 
so in this way rim hook tableaux serve as shorthand for 
the various ways of going from $\lambda$ to $\emptyset$ by 
successively removing rim hooks of prescribed size.

The {\it Murnaghan--Nakayama formula} \cite{Murnaghan,Nakayama} gives, for any
two partitions $\lambda$ and $\mu$ of a positive integer, 
and any sequence $\alpha$ that can be rearranged to $\mu$, 
\begin{equation}
\chi_\lambda(\mu)=\sum_{T\in \T(\lambda,\alpha)} \wt(T).
\end{equation}

\section{Cascades}\label{Sect:Cascades}
By the {\it word} of a partition $\lambda$
we mean the binary sequence $w(\lambda)$
obtained from $\lambda$ by
writing $0$ under each column, $1$ alongside each row, and   
reading clockwise, e.g.\ the word of $(4,2)$ is $001001$:
\[
\includegraphics{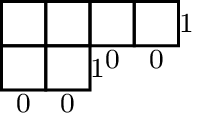}
\]
By the {\it shape} of a binary sequence
$\beta=(\beta_1,\beta_2,\ldots,\beta_k)$ we mean  
the partition
$\sh(\beta)=1^{m_1}2^{m_2}\ldots$
where $m_i$ is the number of
$1$'s in $\beta$ with exactly $i$ $0$'s to the left, e.g.\
both $001001$ and $10010010$ have shape $(4,2)$. The word of a non-empty partition $\lambda$ is
the unique binary sequence of shape $\lambda$ that starts with $0$ and ends with~$1$; the word of the empty partition is the empty sequence.

The standard fact that we require goes back to Com\'et in the 1950's (cf.\ \cite{Comet}) and 
can be stated as follows:
\begin{lemma}\label{Lemma:Comet}
  For any finite binary sequence $\beta$ and integer $k$, 
  the mapping ${\beta'\mapsto \sh(\beta')}$ takes $\mathcal B$,
  the set of $\beta'$ obtainable from $\beta$ by swapping 
  a $0$ with a right-lying $1$ exactly $k$ positions away, bijectively onto
  the set of shapes obtainable from $\sh(\beta)$ by removing a rim hook of size $k$,
  and moreover, the number of rows occupied by the rim hook $\sh(\beta)\setminus \sh(\beta')$
  equals the number of $1$'s lying weakly between the swapped $0$-$1$ pair.\qed
\end{lemma}

For example, if $\lambda$ is the partition $(8,6,4,3)$ and $\rho$ is the rim hook of $\lambda$
shown in \S\ref{Sect:Preliminaries}, and if $\beta=11000101001001$,
so that $\sh(\beta)=\lambda$, then the shape $\lambda\setminus \rho$
corresponds to $\beta'=11010101000001$.

%\subsection*{Cascades}
\subsection{} Our main tool is the following:
\begin{definition}
  A {\it cascade} is a binary matrix 
  $C$ with 
  rows
  $C_i=(C_{i1},C_{i2},\ldots,C_{il})$,  $1\leq i\leq m$, 
  such that
  \begin{enumerate}[1)]
  \item $C_{11}=0$ and $C_{1l}=1$,
  \item for each row $C_i$ with $1\leq i\leq m-1$, there is a
unique pair $a_i<b_i$ such that 
\[
  C_{i a_i}=0,\quad C_{i b_i}=1,\quad
  C_{i+1}=(C_{i\tau(1)},C_{i\tau(2)},\ldots,C_{i\tau(l)})\quad\text{for}\quad
  \tau=\tau_{C,i}=(a_i\ b_i),
\]
  \item $C_{m}=(1,1,\ldots,1,0,0,\ldots,0)$.
  \end{enumerate}

\noindent
The {\it shape} of $C$ is the shape of $C_1$. 

\noindent
The {\it content} of $C$ is the sequence
\[(b_1-a_1,b_2-a_2,\ldots,b_{m-1}-a_{m-1}).\]
  
\noindent
A {\it crossing} in $C$ is a pair
$(i,j)$ such that 
\[1\leq i\leq m-1,\quad  C_{ij}=1,\quad \text{and}\quad a_i<j<b_i.\]
  
\noindent
The {\it weight} of $C$ is defined by
\[\wt(C)=(-1)^{\cross(C)},\quad
\text{where}\quad \cross(C)=\#\{\text{crossings in $C$}\}.\]

\noindent
The {\it permutation} associated to $C$ is 
\[
  \pi_C = \begin{pmatrix}
    1 & 2 & \ldots & k \\
    \sigma_C(i_1) & \sigma_C(i_2) & \ldots &\sigma_C(i_k)
  \end{pmatrix},
\]
where $i_1<i_2<\ldots<i_k$ are the positions of the $1$'s in the first row of $C$,  and
\[\sigma_C=\tau_{C,m-1}\tau_{C,m-2}\ldots \tau_{C,1}.\]
\end{definition}
We denote by $\C(\lambda,\alpha)$ the set of
cascades of shape $\lambda$ and content $\alpha$.

\begin{lemma}\label{Lemma:Corresp}
  The mapping 
  \begin{equation}
    \Theta:C\mapsto
    \Tab(\sh(C_1),\sh(C_2),\ldots, \sh(C_{\#{\text{\rm rows}}(C)}))
    \end{equation}
  takes the set of cascades bijectively onto the set of
  rim hook tableaux, and it preserves shape, content, and weight.
\end{lemma}

\begin{proof}
  This follows from Com\'et's observation in Lemma~\ref{Lemma:Comet}, 
  the standard facts in \S\ref{Sect:Preliminaries} about rim hook tableaux, 
  and the fact that
  there is a unique binary sequence $\beta$ of a given non-empty shape
  such that $\beta$ starts with $0$ and ends with $1$. In particular, 
  \begin{equation}
    \Theta^{-1}:T\mapsto {\rm Mat}(w_\lambda(T_1),w_\lambda(T_2),w_\lambda(T_3),\ldots,w_\lambda(T_{m+1})),
  \end{equation}
  where
$\lambda=\sh(T_1)$,
$m$ is the largest label in $T$,
$w_\lambda(T_i)$ is the sequence obtained from $w(T_i)$ by appending
  to the start $\ell(\lambda)-\ell(T_i)$  many $1$'s and to the end
  $\lambda_1-T_{i1}$ many $0$'s, and
where ${\rm Mat}(r_1,r_2,\ldots,r_k)$
    with $r_i=(r_{i1},r_{i2},\ldots)$ means the matrix $(r_{ij})$.
\end{proof}

\noindent
{\it Example.}
Consider the following cascade $C$:
\begin{equation}\label{Cascade:Example}
  \small\left({\begin{array}{cccccccccccc}
    0 & 0 & 0 & 1 & 0 & 1 & 0 & 0 & 1 & 0 & 0 & 1\\
    0 & 0 & 0 & 1 & 0 & 1 & 0 & 1 & 1 & 0 & 0 & 0\\
    0 & 0 & 0 & 1 & 1 & 1 & 0 & 1 & 0 & 0 & 0 & 0\\
    0 & 1 & 0 & 1 & 1 & 1 & 0 & 0 & 0 & 0 & 0 & 0\\
    1 & 1 & 0 & 0 & 1 & 1 & 0 & 0 & 0 & 0 & 0 & 0\\
    1 & 1 & 0 & 1 & 1 & 0 & 0 & 0 & 0 & 0 & 0 & 0\\
    1 & 1 & 1 & 1 & 0 & 0 & 0 & 0 & 0 & 0 & 0 & 0
  \end{array}}\right)
\end{equation}
The shape is $(8,6,4,3)$,
the content is $(4,4,6,3,2,2)$,
the weight is $(-1)^{1+2+3+1+1+1}$.
The row shapes $\sh(C_k)$ are: 
\[
\includegraphics{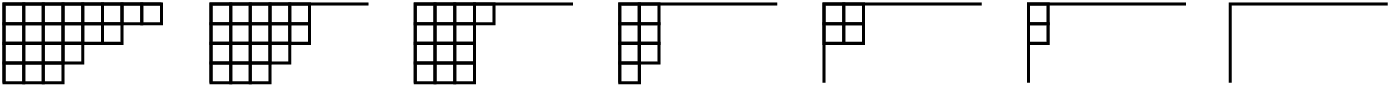}
\]
The corresponding rim hook tableau
$\Tab(\sh(C_1),\sh(C_2),\ldots,\sh(C_7))$ is:
\[
\includegraphics{RH.eps}
\]
The associated permutation 
 $ \pi_C$ is the transposition $(2\ 4)$ in $S_4$.

%\subsection*{Paths}
\subsection{} 
We define a {\it path} in a cascade $C=(C_1,C_2,\ldots,C_m)$
to be a sequence of column positions 
${p=(p_1,p_2,\ldots, p_m)}$, one position $p_i$ for each row $C_i$, 
such that
\[
  C_{1p_1}=1\quad\text{and}\quad 
  p_{i+1}=\tau_{C,i}(p_i)\quad\text{for}\quad
  1\leq i\leq m-1.
\]
We say $p$ {\it starts} at $p_1$ and {\it ends} at $p_m$.
There is exactly one path for each $1$ in the first row of $C$, 
and we agree to always number the paths $p^1,p^2,p^3,\ldots$ according to
relative start position, so that $p^1_1<p^2_1<p^3_1<\dots$.
With this convention, 
\begin{equation}
  \pi_C(i)=p^i_m,\quad i=1,2,\dots.
\end{equation}
By a {\it crossing} of paths $p,p'$ in $C$ we mean a
pair $(i,j)$ with $1\leq i\leq m-1$ such that 
\[p_i=j,\quad
p_i<p_i',\quad \text{and}\quad 
p_{i+1}'<p_{i+1}.\]
\begin{lemma} For a cascade $C$ with paths $p^1,p^2,\ldots,p^k$,
  \begin{equation}\label{cr:decomp}
     \{\text{\rm crossings in $C$}\}
     =
    \dot{\bigcup}_{1\leq i<j\leq k} \left\{\text{\rm crossings of $p^i$ and $p^j$}\right\}.
    \end{equation}
  \end{lemma}
  \begin{proof}
    By comparing definitions.
  \end{proof}

%\subsection*{Diagrams}
\subsection{}
It is often convenient to visualize a cascade by constructing an associated graph.

\begin{definition}
  The {\it diagram} or {\it graph} of a cascade is obtained
  by replacing each $1$ by a node, each 
  $0$ by an empty space ``\,$\cdot$\,'', and then connecting 
  any two nodes $x,y$ that occupy adjacent rows
 and either share a single column or occupy the two columns
 where the two rows differ.
 \end{definition}

 \noindent
 {\it Example.} The diagram of the cascade in \eqref{Cascade:Example} is:
\[
\includegraphics{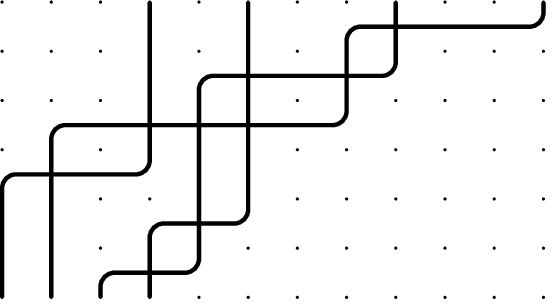}
\]
The paths of the cascade are 
\begin{alignat*}{2}
  p^1&=(4,4,4,4,1,1,1),&\quad
  p^2&=(6,6,6,6,6,4,4),\\
  p^3&=(9,9,5,5,5,5,3),&\quad
  p^4&=(12,8,8,2,2,2,2).
\end{alignat*}
There are $9$ crossings in total, e.g.\ $p^3$ and $p^4$ cross $3$ times.
And the permutation
\[
  \pi_C=
  \begin{pmatrix}
    1 & 2 & 3 & 4 \\
    1 & 4 & 3 & 2
  \end{pmatrix}
\]
can be read off from the diagram by
numbering the nodes in the top row, from left to right, $1,2,\dots$, 
doing the same in the bottom row, and then chasing
through the diagram from top to bottom:
\[
\includegraphics{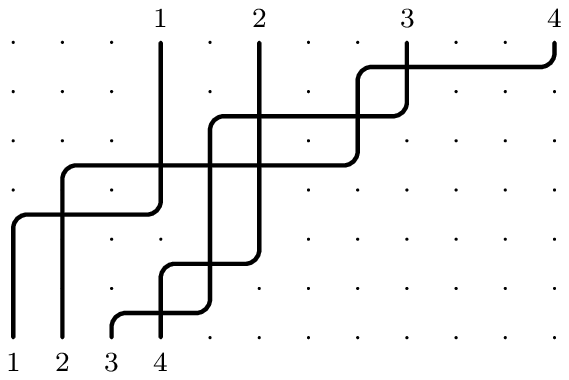}
\]

%\subsection*{The sign of $\boldsymbol{\pi_C}$}
\subsection{}
Denote by $\sgn(\sigma)$ the ${\it sign}$
of a permutation $\sigma$, so that 
\[
  \sgn(\sigma)=(-1)^{\iota(\sigma)},\quad
  \iota(\sigma)=\#\{\text{pairs $i<j$ with  
  $\sigma(j)<\sigma(i)$}\}.
\]

\begin{lemma}\label{Lemma:Wt:Sgn}
  For any cascade $C$, we have
  \begin{equation}
    \wt(C)=\sgn(\pi_C).
  \end{equation}
\end{lemma}

\begin{proof}
  Consider the paths $p^1,p^2,\ldots,p^k$ in $C=(C_1,C_2,\ldots,C_m)$,
  numbered so $\pi_C(i)=p^i_m$, and let
  $\cross(p^i,p^j)$ be the number of crossings of
  $p^i$ and $p^j$, so by \eqref{cr:decomp}, 
  \begin{equation}\label{cr:eq}
    \cross(C)=\sum_{1\leq i<j\leq k} \cross(p^i,p^j).
  \end{equation}
  Fix a pair $i<j$, so $p^i$ starts left of $p^j$.
  If $\pi_C(j)<\pi_C(i)$, then $p^i$ ends to the right of $p^j$,
  so $p^i$ and $p^j$ must have an odd number of crossings; 
  if $\pi_C(i)<\pi_C(j)$, then $p^i$ ends to the left of $p^j$, 
  so $p^i$ and $p^j$ must have an even number of crossings.
  Hence
  \begin{equation}\label{inv:cong}
    \iota(\pi_C)
    \equiv
    \sum_{1\leq i<j\leq k} \cross(p^i,p^j) \pmod{2}.
  \end{equation}
  By \eqref{cr:eq} and \eqref{inv:cong}, we have 
  $\cross(C)\equiv \iota(\pi_C)\pmod{2}$,  so 
  $\wt(C)=\sgn(\pi_C)$.
\end{proof}

As a corollary, 
we have the following useful reformulation of Murnaghan--Nakayama:

\begin{proposition}\label{MN:Lattice-Paths}
  For any two partitions $\lambda$ and $\mu$ of a positive integer, and any
  sequence $\alpha$ that can be rearranged to $\mu$, we have 
\begin{equation}
\chi_\lambda(\mu)=\sum_{C\in \C(\lambda,\alpha)}\wt(C),\quad \wt(C)=(-1)^{\cross(C)}=\sgn(\pi_C),
\end{equation}
where $\C(\lambda,\alpha)$ is
the set of cascades of shape $\lambda$ and content $\alpha$.
\end{proposition}
\begin{proof}
By Lemmas \ref{Lemma:Corresp} and \ref{Lemma:Wt:Sgn}.
\end{proof}

\section{Proof of Theorem~\ref{Thm:main}}\label{Sect:Proof}
\subsection*{An action on cascades}
The main object of this section is to prove the following:
\begin{theorem}\label{Thm:Action}
  Let $\lambda$ be a partition of a positive integer $n$,
  so $\blambda$ is a partition of $d^2n$, 
  and let $\alpha=(\alpha_1,\alpha_2,\ldots, \alpha_m)$
  be a sequence of positive $d$-divisible integers summing to $d^2n$. 

  Define a pairing $\sigma.C$ on $S_d\times \C(\blambda,\alpha)$ by
  \begin{equation}\label{action}
    (\sigma,C)\mapsto C\Phi(\sigma)^{-1},
  \end{equation}
  where $\Phi(\sigma)$ is the block-diagonal matrix 
  \begin{equation*}
  \Phi(\sigma)=
  \begin{pmatrix}
    \phi(\sigma) & & & \\
    & \phi(\sigma) & &\\
    & & \ddots &\\
    & & & \phi(\sigma)
  \end{pmatrix}
\end{equation*}
with $\lambda_1+\ell(\lambda)$ copies of the $d$-by-$d$
permutation matrix $\phi(\sigma)=(\delta_{i\sigma(j)})$
on the diagonal.
   \begin{enumerate}[\rm(i)]
   \item\label{Thm:action} The pairing $\sigma.C$ 
    is an action of $S_d$ on $\C(\blambda,\alpha)$,
    \item\label{Thm:free} the action is free,
    \item\label{Thm:weight} the action is weight-preserving, i.e.\ 
      $\wt(\sigma.C)=\wt(C)$ for all
      $\sigma$ and $C$.
  \end{enumerate}
\end{theorem}

\begin{proof}
  Assume $\C(\blambda,\alpha)\neq \emptyset$. Let $l=\lambda_1+\ell(\lambda)$ and $L=dn+d\ell(\lambda)$.
  
  The word of $\lambda$ starts with $0$, ends with $1$,
  and consists of $\lambda_1$ $0$'s and $\ell(\lambda)$ $1$'s,
  so the sequence $w(\lambda)$ has length $l$.
  The word of $\blambda$ is obtained by replacing in $w(\lambda)$ each
  $0$ by $d$ consecutive $0$'s and each $1$ by $d$ consecutive $1$'s, so
  $w(\blambda)$ starts with $d$ $0$'s, ends with $d$ $1$'s, has length $L$,
  and writing $w(\blambda)=(w_1,w_2,\ldots, w_L)$, 
  \begin{equation}\label{block:eq}
    w_{1+dk}=w_{2+dk}=\ldots =w_{d+dk},\quad 0\leq k\leq L/d-1.
  \end{equation}
  In particular, each $C\in\C(\blambda,\alpha)$
  has $L$ columns, so the matrix multiplication
  on the right-hand side of \eqref{action} makes sense,
  and multiplying $C$ on the right by $\Phi(\sigma)^{-1}$ permutes
  the first $d$ columns of $C$, the next $d$ columns of $C$,
  and so on:
  denoting by ${\rm Col}_i(C)$ the $i$-th column of $C$, we have
  \begin{equation}\label{perm:eq}
    {\rm Col}_{i+dk}(C)={\rm Col}_{\sigma(i)+dk}(\sigma.C)
  \end{equation}
  for $1\leq i\leq d$ and $0\leq k \leq L/d-1$.
  
  \subsection*{\eqref{Thm:action}}
Fix $C\in\C(\blambda,\alpha)$ and $\sigma\in S_d$.
  Let
  $C'=\sigma.C$. 
  By \eqref{block:eq} and~\eqref{perm:eq}, 
  \begin{equation}\label{first:row}
    C_1'=C_1.
  \end{equation}
  The last row of $C$ is $C_m=(1,\ldots,1,0,\ldots,0)$,
  with $d\ell(\lambda)$ $1$'s, so by  \eqref{perm:eq},
  \begin{equation}\label{last:row}
    C_m'=C_m.
  \end{equation}
  By \eqref{first:row}, \eqref{last:row}, and $C$ being a cascade, $C'$
  satisfies the first and third cascade conditions.

  Let $C_i'$ and $C_{i+1}'$ be two consecutive rows in $C'$.
  Since $C$ is a cascade, the rows 
  $C_i$ and $C_{i+1}$ differ in exactly two positions, $a_i$ and $b_i$ with $a_i<b_i$, and 
  \begin{equation*}
    C_{i,a_i}=0,\quad C_{i,b_i}=1,\quad
    C_{i+1,a_i}=1,\quad C_{i+1,b_i}=0.
  \end{equation*}
  Since the difference $\alpha_i=b_i-a_i$ is positive and divisible by $d$, 
  \begin{equation}\label{ab}
    a_i=r_i+ds_i\quad\text{and}\quad b_i=r_i+dt_i
  \end{equation}
  for some non-negative integers $r_i,s_i,t_i$ with
  $1\leq r_i\leq d$ and $s_i< t_i$. 
  Setting  
  \begin{equation}\label{apbp}
    a_i'=\sigma(r_i)+ds_i\quad\text{and}\quad b_i'=\sigma(r_i)+dt_i, 
  \end{equation}
  and using \eqref{perm:eq}, we have that 
  $C_i'$ and $C_{i+1}'$ differ in exactly positions $a_i'$ and $b_i'$, and
   \begin{equation*}
    C_{i,a_i'}'=0,\quad C_{i,b_i'}'=1,\quad
    C_{i+1,a_i'}'=1,\quad C_{i+1,b_i'}'=0.
  \end{equation*}
  Since $s_i<t_i$, we also have that $a_i'<b_i'$.
  So $C'$ satisfies the second condition of a cascade. Hence $C'$ is a cascade.
  
  By \eqref{first:row}, the shape of the cascade $C'$ is $\blambda$. 
  The content of $C'$ is $(b_1'-a_1',b_2'-a_2',\ldots)$,
  which by \eqref{ab} and \eqref{apbp} equals $\alpha$.  
  So $C'\in\C(\blambda,\alpha)$. This concludes the proof of~\eqref{Thm:action}.
  
  \subsection*{\eqref{Thm:free}}
  Let
  $z_i(C)$ be the number of $0$'s in the $i$-th column of a cascade $C\in \C(\blambda,\alpha)$.
  Let 
  \begin{equation}
    z(C)=(z_1(C),z_2(C),\ldots, z_d(C)).
  \end{equation}
  By the cascade conditions, and the positivity and $d$-divisibility of the $\alpha_i$'s, we have 
  \begin{equation}\label{z}
    z_i(C)\neq z_j(C)\quad\text{for}\quad 1\leq i<j\leq d.
  \end{equation}
  By \eqref{perm:eq},
  \begin{equation}\label{z:perm}
    z(\sigma.C)=(z_{\sigma^{-1}(1)}(C),z_{\sigma^{-1}(2)}(C),\ldots,z_{\sigma^{-1}(d)}(C)).  
  \end{equation}
  From \eqref{z} and \eqref{z:perm}, for each $C\in \C(\blambda,\alpha)$, we have 
  \begin{equation}
    \sigma.C=C\text{ if and only if }\sigma=1.
  \end{equation}
  This concludes the proof of \eqref{Thm:free}.
   
  \subsection*{\eqref{Thm:weight}}
  Fix a cascade $C\in\C(\blambda,\alpha)$ and
  a permutation $\sigma\in S_d$, so
  $\sigma.C\in \C(\blambda,\alpha)$ by~\eqref{Thm:action}.
  Let $p^1,p^2,\ldots,p^{d\ell(\lambda)}$ be the paths in $C$, so  
  $p^1_1<p^2_1<\dots$ and
  \begin{equation}\label{pi:img}
    \pi_C(i)=p^i_m,
  \end{equation}
  and let $q^1,q^2,\ldots,q^{d\ell(\lambda)}$ be the paths in $\sigma.C$,
  so $q^1_1<q^2_1<\dots$ and
  \begin{equation}\label{s:pi:img}
    \pi_{\sigma.C}(i)=q^i_m.
  \end{equation}
  Let $\gamma$ be the permutation in $S_L$ given by 
  \begin{equation}
    \gamma(i+dk)=\sigma(i)+dk,\quad 1\leq i\leq d,\quad 0\leq k\leq L/d-1.
  \end{equation}
  By \eqref{perm:eq}, the sequences 
  \begin{equation}
    \sigma.p^i
    =
    (\gamma(p^i_1),\gamma(p^i_2),\ldots,\gamma(p^i_m)),
    \quad
    1\leq i\leq d\ell(\lambda),
  \end{equation}
  are the paths of $\sigma.C$, in some order.
  Let $\omega$ be the permutation in $S_{d\ell(\lambda)}$ given by 
  \begin{equation}
    \omega(i+dk)=\sigma(i)+dk,\quad 1\leq i\leq d,\quad 0\leq k\leq \ell(\lambda)-1.
  \end{equation}
  Then by \eqref{block:eq}, for each $i$, 
  \begin{equation}\label{qsp}
    q^{\omega(i)}=\sigma.p^i.
  \end{equation}
  Since $C_m=(1,\ldots,1,0,\ldots,0)$ with $d\ell(\lambda)$ $1$'s,
  we also have $\gamma(p^i_m)=\omega(p^i_m)$, so
\begin{equation}\label{qpe}
  q^{\omega(i)}_m=\omega(p^i_m).
\end{equation}
By \eqref{pi:img}, \eqref{s:pi:img}, and \eqref{qpe}, 
the permutation $\pi_{\sigma.C}$
takes $\omega(i)$ to $\omega(\pi_C(i))$ for each~$i$, i.e.\
\begin{equation}
\pi_{\sigma.C}=\omega \pi_C\omega^{-1}.
\end{equation}
So $\pi_{\sigma.C}$ and $\pi_C$ have the same sign.
By Lemma~\ref{Lemma:Wt:Sgn}, we conclude that 
\begin{equation}
\wt(\sigma.C)=\wt(C)
\end{equation}
for all $\sigma\in S_d$ and $C\in\C(\blambda,\alpha)$. This concludes the proof of \eqref{Thm:weight} and Theorem~\ref{Thm:Action}.
\end{proof}

It is worth remarking that Theorem~\ref{Thm:Action} and Lemma~\ref{Lemma:Corresp} together
give a weight-preserving free action on rim hook tableaux:
\begin{corollary}\label{Cor:Action}
  For any partition $\lambda$ of a positive integer $n$, and
  any sequence $\alpha$ of positive $d$-divisible integers summing
  to $d^2n$, there is a well-defined action of $S_d$ on
  $\T(\blambda,\alpha)$ given by
  $\sigma.T=\Theta(\sigma.\Theta^{-1}(T))$,  
  and this action is both free and weight-preserving.\qed
\end{corollary}
%\subsection*{Example}
\noindent {\it Example.}
With $d=3$ and $\lambda=(3,2)$,
the following shows 
an $S_d$-orbit of a cascade $C$ and corresponding rim hook tableau $T$
of shape $\blambda$ and content $(3,3,6,6,3,3,6,9,3)$.
%\vspace{-.5in}
\newpage
\begin{center}
%\raisebox{-.7in}[3.35in][2in]{\makebox{%
\begin{tabular}{@{\hspace{.065\textwidth}}c@{\hspace{.15\textwidth}}c@{\hspace{.15\textwidth}}c@{\hspace{.065\textwidth}}}
% \begin{tabular}{@{\hspace{.07\textwidth}}c@{\hspace{.15\textwidth}}c@{\hspace{.15\textwidth}}c@{\hspace{.07\textwidth}}}
  \toprule
  $\sigma$ & diagram of $\sigma.C$ & $\sigma.T$ \\
  \midrule
  \ink{$1$}           & \ink{\includegraphics[scale=.88]{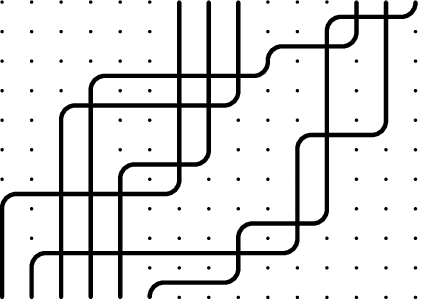}} & \ink{\includegraphics{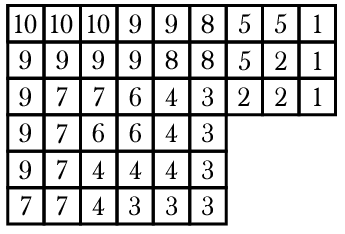}}\\
  \ink{$(1\, 2)$}     & \ink{\includegraphics[scale=.88]{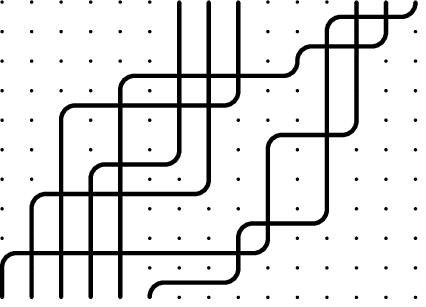}} & \ink{\includegraphics{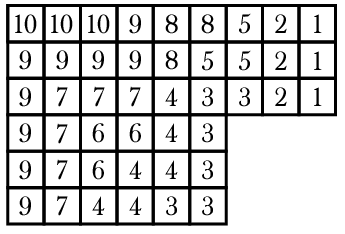}}\\
  \ink{$(1\, 3)$}     & \ink{\includegraphics[scale=.88]{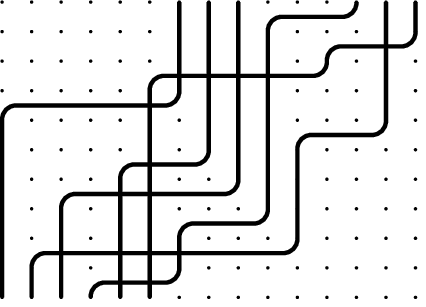}} & \ink{\includegraphics{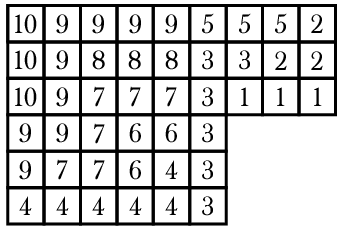}}\\
  \ink{$(2\, 3)$}     & \ink{\includegraphics[scale=.88]{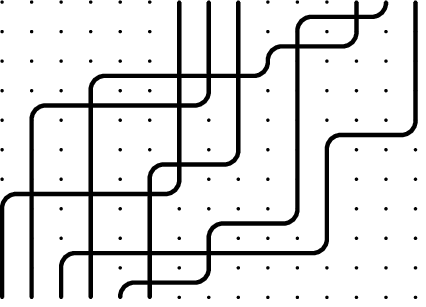}} & \ink{\includegraphics{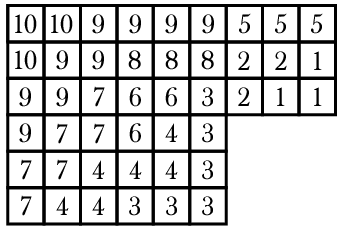}}\\
  \ink{$(1\, 2\, 3)$} & \ink{\includegraphics[scale=.88]{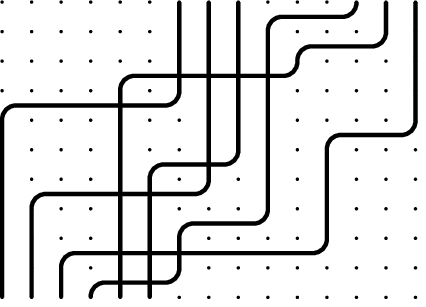}} & \ink{\includegraphics{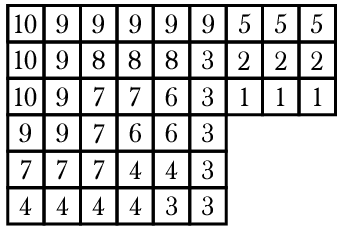}}\\
  \ink{$(1\, 3\, 2)$} & \ink{\includegraphics[scale=.88]{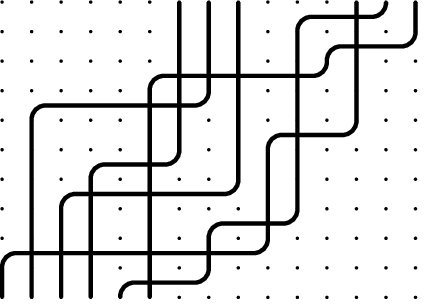}} & \ink{\includegraphics{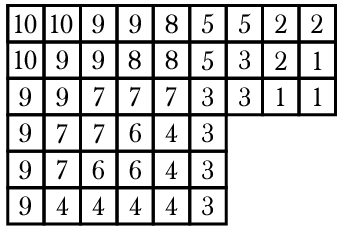}}\\
  \bottomrule
\end{tabular}%}}
\end{center}
\newpage

\subsection*{Proof of Theorem~\ref{Thm:main}}
  For \eqref{General:Mod:Vanish}, let $\lambda$ be a partition of a positive integer $n$, and let $\mu$ be
  a partition of $dn$. By Proposition~\ref{MN:Lattice-Paths}, we have
  \[
    \chi_{\blambda}(d.\mu)=\sum_{C\in \C(\blambda,d.\mu)}\wt(C),
  \]
  and by Theorem~\ref{Thm:Action} there exists a weight-preserving free action of $S_d$
  on $\C(\blambda,d.\mu)$.
  So $\chi_{\blambda}(d.\mu)$ is divisible by $d!$. This completes the proof of \eqref{General:Mod:Vanish}.

  \eqref{Mod:Vanish} is a special case of \eqref{General:Mod:Vanish}: let $\lambda$ and $\mu$ be partitions of a positive integer $n$,
  write $\mu=1^{m_1}2^{m_2}\ldots n^{m_n}$, and define $\nu=1^{dm_1}2^{dm_2}\ldots n^{dm_n}$,
  so that $\nu$ is a partition of $dn$ with $d.\nu=\bmu$, and hence by \eqref{General:Mod:Vanish},
  $\chi_{\blambda}(\boldsymbol\mu)$ is divisible by $d!$.

  For \eqref{Ord:Vanish},
  let $\lambda$ and $\mu$ be partitions of an integer $n$ not divisible by~$d$. 
  Suppose that there exists a cascade $C\in \C(\blambda,d^2.\mu)$, 
  let $D$ be the matrix with columns
  \[{\rm Col}_1(C),{\rm Col}_{d+1}(C),{\rm Col}_{2d+1}(C),\dots,\]
  occurring in that order, and let $C'$ be the matrix obtained
  from $D$ by deleting redundant rows.
  Then 
  $C'\in \C(\lambda,d.\mu')$ for some partition $\mu'$, 
  hence $n=d|\mu'|$. So
  ${\C(\blambda,d^2.\mu)=\emptyset}$,
    hence by Proposition~\ref{MN:Lattice-Paths}, $\chi_{\blambda}(d^2.\mu)$ equals $0$.
    \qed
%%%%%%%%%%%%%%%%%%%%%%%%%%%%%%%%%%%%%%%%%%%%%%%%%%%%%%%%%%%%%%%%%%
%%%%%%%%%%%%%%%%%%%%%%%%%%%%%%%%%%%%%%%%%%%%%%%%%%%%%%%%%%%%%%%%%%

\end{document}